\documentclass[a4paper,11pt,reqno]{amsart}
\usepackage{amssymb}
\usepackage{amsmath,enumerate}
\usepackage{graphicx}
\usepackage{tikz}
\usetikzlibrary {arrows.meta}

\usepackage[top=30truemm,bottom=30truemm,left=20truemm,right=20truemm]{geometry}
\allowdisplaybreaks[4]

\newtheorem{Def}{Definition}[section]
\theoremstyle{remark}
\newtheorem{Rem}[Def]{Remark}

\newtheorem*{Ack}{Acknowledgments}
\theoremstyle{plain}
\newtheorem{Th}[Def]{Theorem}
\newtheorem{Prop}[Def]{Proposition}

\newtheorem{Cor}[Def]{Corollary}
\newtheorem{Fact}[Def]{Fact}

\newcommand{\Z}{\mathbb{Z}}

\newcommand{\C}{\mathbb{C}}
\renewcommand{\P}{\mathbb{P}}
\renewcommand{\H}{\mathbb{H}}

\newcommand{\CL}{\mathcal{L}}

\newcommand{\CO}{\mathcal{O}}

\newcommand{\al}{\alpha }
\newcommand{\be}{\beta }
\newcommand{\ga}{\gamma }

\newcommand{\Ga}{\Gamma }

\newcommand{\vth}{\vartheta }

\newcommand{\vph}{\varphi }
\newcommand{\vsi}{\varsigma }
\newcommand{\la}{\lambda }

\newcommand{\Om}{\Omega }
\newcommand{\si}{\sigma }
\newcommand{\na}{\nabla }

\newcommand{\ot}{\otimes }
\newcommand{\op}{\oplus }
\newcommand{\bu}{\bullet}
\newcommand{\we}{\wedge}

\newcommand{\DS}{\displaystyle }

\newcommand{\tpi}{2\pi \sqrt{-1}}
\newcommand{\pii}{\pi \sqrt{-1}}

\newcommand{\bfe}{\mathbf{e}}

\newcommand{\h}{{\rm h}}
\newcommand{\ch}{{\rm ch}}

\DeclareMathOperator{\im}{Im}
\DeclareMathOperator{\pr}{pr}

\def\tp#1{\mathord{\mathopen{{\vphantom{#1}}^t}#1}} 

\makeatletter
\@addtoreset{equation}{section}

\makeatother

\title[The Wirtinger-type integral for a genus two curve]{
  The Wirtinger-type integral for a genus two curve
}
\author[Y. Goto]{Yoshiaki Goto}
\address[Goto]{
  Otaru University of Commerce, 
  3-5-21, Midori, Otaru, Hokkaido, 047-8501, Japan
}
\email{goto@res.otaru-uc.ac.jp}
\date{\today}

\begin{document}
\begin{abstract}
  The Wirtinger integral is one of the integral representations of the Gauss hypergeometric function. 
  Its integrand can be regarded as a multivalued function on an elliptic curve. 
  In this paper, we study an analogue of the Wirtinger integral on a hyperelliptic curve of genus two, 
  introduced by Mizutani and Watanabe. 
  We investigate the associated twisted homology and cohomology groups using 
  the hyperelliptic involution and intersection forms. 
\end{abstract}

\maketitle

\section{Introduction}
The Gauss hypergeometric function ${}_2F_1(a,b,c;z)$ admits an integral representation 
\begin{align*}
  {}_2F_1(a,b,c;z) =\frac{\Ga (c)}{\Ga (a) \Ga (c-a)}
    \int_{0}^{1} t^{a} (1-t)^{c-a} (1-zt)^{-b} \frac{dt}{t(1-t)}.
\end{align*}
By using the Jacobi theta functions $\vth_1$, $\vth_2$, $\vth_3$, $\vth_4$ (e.g., \cite{Erdelyi}) and 
setting $t=(\vth_3(0)^2 \vth_1(u)^2)/(\vth_2(0)^2\vth_4(u)^2)$, 
we obtain another integral representation 
\begin{align}
  \nonumber
  {}_2F_1(a,b,c;\frac{\vth_2(0)^4}{\vth_3(0)^4}) 
  &=\frac{2\pi \Ga (c)}{\Ga (a) \Ga (c-a)}
    \vth_2(0)^{2-2\ga} \vth_3(0)^{2\al +2\be} \vth_4(0)^{2\ga-2\al -2\be} \\
  \label{eq:Wirtinger-integral}
  &\cdot \int_0^{\frac{1}{2}} \vth_1(u)^{2\al-1} \vth_2(u)^{2\ga-2\al-1} \vth_3(u)^{-2\be+1} \vth_4(u)^{2\be -2\ga +1} du ,
\end{align}
which is called the Wirtinger integral \cite{Wirtinger}. 
For further details, see \cite{Watanabe-elliptic-homology-cohomology}. 
Here, for simplicity, we write $\vth_1(u)$ for $\vth_1(u;\tau)$ ($u,\tau\in \C$, $\im (\tau)>0$). 
Note that the integrand is a multivalued function on the elliptic curve $\C/(\Z +\Z \tau)$. 

To study various types of hypergeometric functions, twisted homology and cohomology groups play 
an important role. 
Twisted (co)homology groups are defined as those with coefficients in local systems associated with 
multivalued functions. 
Many studies of twisted (co)homology groups use multivalued functions on projective spaces. 
In contrast, we consider a generalization of twisted (co)homology theory to non-projective spaces. 
As a first example, 
twisted homology and cohomology groups associated with the integral (\ref{eq:Wirtinger-integral}) are studied
in \cite{Watanabe-elliptic-homology-cohomology}, \cite{Watanabe-wirtinger-diff-eq}. 
Recently, in \cite{G-Shibukawa}, theses results are reconsidered and interpreted in terms of intersection forms. 

In \cite{G-Matsubara-Mitsui}, we study twisted homology and cohomology groups with an Abelian covering $p:X\to Y$. 
Using the intersection pairings, the twisted (co)homology groups of $X$ can be understood 
in terms of those of $Y$. 
The results of \cite{G-Shibukawa} for the Wirtinger integral can also be interpreted in this framework 
via the double covering $\C/(\Z +\Z \tau) \to \P^1$. 
Motivated by this observation, 
we expect to obtain similar results for hyperelliptic curves. 
Accordingly, in this paper, we consider hypergeometric integrals on a hyperelliptic curve of genus two. 

The Lauricella hypergeometric function $F_D$ of three variables admits an integral representation
\begin{align}
  \nonumber
  &F_D (a,b_1,b_2,b_3,c;z_1,z_2,z_3) \\
  \label{eq:FD-int-rep}
  &=\frac{\Ga(c)}{\Ga(a) \Ga(c-a)} \int_{1}^{\infty}
  t^{b_1+b_2+b_3-c} (t-1)^{c-a} (t-z_1)^{-b_1}(t-z_2)^{-b_2}(t-z_3)^{-b_3} \frac{dt}{t-1} .
\end{align}
As an analogue of the Wirtinger integral, we consider the pullback of this integral 
under the double covering $\bar{X} \to \P^1$, where $\bar{X}$ is a hyperelliptic curve of genus two. 
In \cite{Mizutani-Watanabe-J}, \cite{Mizutani-Watanabe-E}\footnote{
\cite{Mizutani-Watanabe-J} is written in Japanese. 
\cite{Mizutani-Watanabe-E} is its English version. 
}, 
a certain relation in the twisted homology group is studied using the Abel-Jacobi map and theta functions. 
However, the meaning of this relation from the viewpoint of twisted (co)homology theory 
does not seem to have been clarified.
Applying the results of \cite{G-Matsubara-Mitsui}, 
we study a precise structure of the twisted (co)homology groups and 
interpret the relation in terms of the intersection forms. 

Twisted (co)homology theory for punctured Riemann surfaces is studied in \cite{Watanabe-punctured}, and 
the corresponding intersection theory is investigated in \cite{Pokraka-Ren-Rodriguez}. 
Our setting is covered by the framework of \cite{Pokraka-Ren-Rodriguez}. 
However, we choose six special points as the punctures of the hyperelliptic curve, 
whereas in \cite{Pokraka-Ren-Rodriguez} the punctures are taken to be generic. 
This specific choice leads to a simpler structure of the twisted (co)homology groups. 

This paper is organized as follows. 
In Section \ref{section:preliminaries}, 
we review the Wirtinger-type integral of genus two which is defined 
in \cite{Mizutani-Watanabe-J}, \cite{Mizutani-Watanabe-E}. 
We also review basic facts on twisted homology and cohomology groups. 
In Section \ref{section:involution}, 
we apply the results of \cite{G-Matsubara-Mitsui} to study twisted (co)homology groups. 
Using the hyperelliptic involution, we decompose these groups into eigenspaces, 
and show that this decomposition is ``compatible'' with the intersection forms. 
In Sections \ref{section:cohomology} and \ref{section:homology}, 
we construct suitable bases of twisted (co)homology groups consisting of eigenvectors, 
and compute their intersection matrices. 
In particular, in Section \ref{section:homology}, 
the construction of the basis relies on the results of
\cite{Mizutani-Watanabe-J}, \cite{Mizutani-Watanabe-E}, which are obtained using 
theta functions. 
Finally, by combining the decomposition with the intersection forms, 
we reinterpret and generalize the main result 
of \cite{Mizutani-Watanabe-J}, \cite{Mizutani-Watanabe-E}.

\section{Preliminaries}\label{section:preliminaries}
\subsection{Hyperelliptic curve of genus two}
Let $\bar{X}$ be a hyperelliptic curve of genus two defined by 
\begin{align*}
  y^2=x(x-1)(x-z_1)(x-z_2)(x-z_3), 
\end{align*}
where $z_1,z_2,z_3 \neq 0,1$ are pairwise distinct. We set 
\begin{align*}
  P_0=(0,0),\ P_1=(z_1,0),\ P_2=(z_2,0),\ P_3=(z_3,0),\ P_4=(1,0)
\end{align*}
and write $P_5$ for the point at infinity. 
We consider two varieties 
\begin{align*}
  X:=\bar{X}-\{ P_0,\dots ,P_5 \} ,\quad 
  Y:=\P_t^1 -\{ 0,z_1,z_2,z_3,1,\infty \} , 
\end{align*}
which have an unramified covering and an involution
\begin{align*}
  \pr:X\ni (x,y)\mapsto x\in Y ,\quad 
  \iota :X \ni (x,y)\mapsto (x,-y)\in X ,
\end{align*}
respectively. 
By taking the pullback of (\ref{eq:FD-int-rep}) via $\pr$, we obtain an integral
\begin{align}
  \label{eq:Wint-x}
  \int_{\text{cycle}}
  x^{b_1+b_2+b_3-c} (x-1)^{c-a} (x-z_1)^{-b_1}(x-z_2)^{-b_2}(x-z_3)^{-b_3} \frac{dx}{x-1}
\end{align}
on $X$, which we call the \emph{Wirtinger-type integral of genus 2}. 
For simplicity, we set 
\begin{align*}
  &z_0:=0,& && && && &z_4:=1,& &z_5:=\infty,& \\
  &c_0 :=\frac{b_1+b_2+b_3-c}{2} ,&
  &c_1 :=\frac{-b_1}{2},&
  &c_2 :=\frac{-b_2}{2},&
  &c_3 :=\frac{-b_3}{2},&
  &c_4 :=\frac{c-a}{2},&
  &c_5 :=\frac{a}{2}.
\end{align*}
Then the integrand of (\ref{eq:FD-int-rep}) is written as 
$T:=\prod_{j=0}^4(x-z_j)^{2c_j}$. 
Throughout this paper, we assume 
\begin{align*}
  4c_j \not\in \Z \qquad (j=0,\dots ,5). 
\end{align*}

\subsection{Expression by theta functions}
Using theta functions and the Abel-Jacobi map, we rewrite the integral (\ref{eq:FD-int-rep}). 
Following \cite{Mizutani-Watanabe-J}, \cite{Mizutani-Watanabe-E}, 
we briefly review this fact. 
For $z=(z_1,z_2)\in \C^2$, $a=(a_1,a_2),b=(b_1,b_2)\in \{ 0,1 \}^2$ 
and $\tau \in \H_2:=\{ \tau \in M(2;\C) \mid \tp{\tau} =\tau ,\ \im(\tau)>0 \}$, 
we set
\begin{align*}
  \vth\begin{bmatrix}a \\ b\end{bmatrix} (z;\tau)
  =\sum_{n\in \Z^2} \bfe\left( \frac{1}{2}\left(n+ \frac{a}{2} \right) \tau \tp{\left( n+ \frac{a}{2} \right) }
  +\left(n+ \frac{a}{2} \right) \tp{\left( z+ \frac{b}{2} \right) } \right) ,
\end{align*}
where $\bfe(u)=e^{\tpi u}$. 

We take a symplectic basis $\{ \la_1,\la_2,\la_3,\la_4 \}$ of $H_1(\bar{X};\Z)$ such that 
$\la_1 \cdot \la_3 =\la_2 \cdot \la_4 =1$, $\la_i \cdot \la_j =0$ ($|i-j|\neq 2$) where 
$\la_i \cdot \la_j$ means the intersection number. 
We also choose holomorphic 1-forms $\zeta_1$, $\zeta_2$ such that 
\begin{align*}
  \left( \int_{\la_i} \zeta_j \right)_{i,j=1,2} =E_2 ,\quad 
  \tau :=\left( \int_{\la_{i+2}} \zeta_j \right)_{i,j=1,2} \in \H_2 ,
\end{align*}
where $E_2$ denotes the $2\times 2$ identity matrix. 
The Abel-Jacobi map $AJ$ is defined by 
\begin{align*}
  AJ: \bar{X} \to \C^2 /(\Z^2 +\tau \Z^2);\quad 
  P \mapsto \left( \int_{P_0}^{P} \zeta_1 , \int_{P_0}^{P} \zeta_2 \right) .
\end{align*}
We set 
$\DS 
\begin{Bmatrix}a_1 & a_2 \\ b_1 & b_2\end{Bmatrix}
=\begin{Bmatrix}a_1 & a_2 \\ b_1 & b_2\end{Bmatrix}(P)
:=\vth\begin{bmatrix}a_1 & a_2 \\ b_1 & b_2 \end{bmatrix} (AJ(P);\tau)$. 
Then the integrand $T$ of (\ref{eq:FD-int-rep}) is a constant multiple of\footnote{
In \cite{Mizutani-Watanabe-J}, \cite{Mizutani-Watanabe-E}, 
the point $P_2$ is specified as $P_2=(1,0)\in \bar{X}$. 
Though it is different from our setting ($P_4=(1,0)$), this difference does not effect the Abel-Jacobi map. 
Thus, we can use the same expression as \cite{Mizutani-Watanabe-J}, \cite{Mizutani-Watanabe-E}.
} 
\begin{align}
  \nonumber
  &\begin{Bmatrix}0&0 \\ 0&0\end{Bmatrix}^{c_2+c_4}
    \begin{Bmatrix}0&1 \\ 0&0\end{Bmatrix}^{c_3+c_4}
    \begin{Bmatrix}0&0 \\ 0&1\end{Bmatrix}^{c_2+c_3}
    \begin{Bmatrix}1&0 \\ 0&0\end{Bmatrix}^{c_3+c_5} 
    \begin{Bmatrix}0&0 \\ 1&0\end{Bmatrix}^{c_1+c_3} \\
  \nonumber
    & \cdot 
    \begin{Bmatrix}1&1 \\ 0&0\end{Bmatrix}^{c_2+c_5}
    \begin{Bmatrix}1&0 \\ 1&0\end{Bmatrix}^{c_0+c_3}
    \begin{Bmatrix}0&0 \\ 1&1\end{Bmatrix}^{c_1+c_4}
    \begin{Bmatrix}1&0 \\ 0&1\end{Bmatrix}^{c_4+c_5}
    \begin{Bmatrix}0&1 \\ 1&0\end{Bmatrix}^{c_1+c_2} \\
  \label{eq:theta-expression}
    & \cdot 
    \begin{Bmatrix}1&1 \\ 1&0\end{Bmatrix}^{c_0+c_2}
    \begin{Bmatrix}1&1 \\ 0&1\end{Bmatrix}^{c_0+c_1}
    \begin{Bmatrix}0&1 \\ 1&1\end{Bmatrix}^{c_0+c_5}
    \begin{Bmatrix}1&0 \\ 1&1\end{Bmatrix}^{c_0+c_4}
    \begin{Bmatrix}1&1 \\ 1&1\end{Bmatrix}^{c_1+c_5} .
\end{align}

\subsection{Twisted homology and cohomology groups}
Let $T_0$ be a multivalued function on $Y$ defined by $T_0=\prod_{j=0}^4(t-z_j)^{2c_j}$. 
We define local systems $\CL, \CL^{\vee}$ on $X$ and $\CL_0, \CL_0^{\vee}$ on $Y$ by
\begin{align*}
  \CL :=\C T,\ \CL^{\vee}:=\C T^{-1}, \ 
  \CL_0 :=\C T_0, \ \CL_0^{\vee} :=\C T_0^{-1} .
\end{align*}
Note that the local monodromy of $T$ around $P_j$ is given by $\bfe(4c_j)=e^{8\pii c_j}$. 
We also define connections $\na,\na^{\vee}:\CO_X \to \Om_X^1$ and $\na_0,\na_0^{\vee}:\CO_Y \to \Om_Y^1$ by 
\begin{align*}
  \na :=d+d\log T \we, \ \na^{\vee} :=d-d\log T \we ,\ 
  \na_0 :=d+d\log T_0 \we, \ \na_0^{\vee} :=d-d\log T_0 \we . 
\end{align*}
We note that $\CL$ is expressed as $\CL=\ker (\na^{\vee})$. 
We consider the twisted homology groups and twisted de Rham cohomology groups\footnote{
  As in \cite{AK} and \cite{Watanabe-punctured}, the cohomology group with coefficients in 
  the local system $\CL=\ker (\na^{\vee})$ is isomorphic to the twisted de Rham cohomology group $H^k(X;\na)$. 
}: 
\begin{align*}
  &H_k(X;\CL) ,\ H_k(X;\CL^{\vee}),\  H_k(Y;\CL_0) ,\  H_k(Y;\CL_0^{\vee}) ,\\
  &H^k(X;\na) ,\ H^k(X;\na^{\vee}),\  H^k(Y;\na_0) ,\  H^k(Y;\na_0^{\vee}) .  
\end{align*}
For details, see e.g. \cite{AK} and \cite{Watanabe-punctured}. 
For example, we have $H^1 (X;\na) = \Om^1_X(X)/\na \CO_X(X)$. 
\begin{Fact}[\cite{AK}, \cite{Watanabe-punctured}]\label{fact:vanish-dim}
  Under the assumption $4c_j\not\in \Z$ for $j=0,\dots,5$,  
  the twisted homology and cohomology groups vanish except in degree one. 
  We have 
  \begin{align*}
    &\dim H_1(X;\CL) =\dim H_1(X;\CL^{\vee}) =\dim H^1(X;\na)=\dim H^1(X;\na^{\vee}) =8,\\
    &\dim H_1(Y;\CL_0) =\dim H_1(Y;\CL_0^{\vee}) =\dim H^1(Y;\na_0)=\dim H^1(Y;\na_0^{\vee}) =4.
  \end{align*}
\end{Fact}
Recall that the intersection forms are naturally defined: 
\begin{align*}
  &\langle \bu ,\bu \rangle_{\h} : H_1(X;\CL) \times H_1(X;\CL^{\vee}) \longrightarrow \C ,\quad& 
  &\langle \bu ,\bu \rangle_{\h,0} : H_1(Y;\CL_0) \times H_1(Y;\CL_0^{\vee}) \longrightarrow \C ,&\\
  &\langle \bu ,\bu \rangle_{\ch} : H^1(X;\na) \times H^1(X;\na^{\vee}) \longrightarrow \C ,\quad &
  &\langle \bu ,\bu \rangle_{\ch,0} : H^1(Y;\na_0) \times H^1(Y;\na_0^{\vee}) \longrightarrow \C .&
\end{align*}
By taking integrals, we have the period pairings 
$H^1(X;\na) \times H_1(X;\CL) \ni (\vph ,\si) \mapsto \int_{\si} T \vph \in \C$ and so on. 
For a twisted cycle $\si$ with coefficients in $\CL$ (resp. $\CL_0$), 
we can construct a cycle $\si^{\vee}$ with coefficients in $\CL^{\vee}$ (resp. $\CL_0^{\vee}$) by 
replacing $c_j$ with $-c_j$.

\subsection{Bases of $H_1(Y;\CL_0)$ and $H^1(Y;\na_0)$}
Let $\si_{j,j+1}\ (j=0,\dots ,4)$ be a twisted cycle in $Y$ obtained as
the regularization of a locally finite path from $z_j$ to $z_{j+1}$; 
for a detailed construction, see \cite{KY}.  
Here, the branch of $T_0$ on $(1,\infty)$ is assumed to be $\arg (t-z_j)=0$, and 
those on other cycles are defined by analytic continuation along the lower half-plane.
It is well-known that $\si_{01}$, $\si_{12}$, $\si_{23}$, $\si_{34}$ form a basis of $H_1(Y;\CL_0)$. 
Their intersection matrix is given as follows (cf. \cite{KY}):
\begin{align}
  \label{eq:int-mat-homology-Y}
  \begin{pmatrix}
    \frac{1-\bfe(2c_0+2c_1)}{(1-\bfe(2c_0))(1-\bfe(2c_1))}&\frac{1}{1-\bfe(2c_1)}&0&0 \\
    \frac{\bfe(2c_1)}{1-\bfe(2c_1)}&\frac{1-\bfe(2c_1+2c_2)}{(1-\bfe(2c_1))(1-\bfe(2c_2))}
    &\frac{1}{1-\bfe(2c_2)}&0 \\
    0&\frac{\bfe(2c_2)}{1-\bfe(2c_2)}
    &\frac{1-\bfe(2c_2+2c_3)}{(1-\bfe(2c_2))(1-\bfe(2c_3))}&\frac{1}{1-\bfe(2c_3)} \\
    0&0&\frac{\bfe(2c_3)}{1-\bfe(2c_3)}&\frac{1-\bfe(2c_3+2c_4)}{(1-\bfe(2c_3))(1-\bfe(2c_4))} 
  \end{pmatrix}
\end{align}
It is also known that logarithmic $1$-forms on $Y$ defined by 
\begin{align*}
  &\phi_1 =\frac{dt}{t-1}=d\log (t-z_4),\quad   
  \phi_2 =\frac{dt}{t(t-1)}=d\log \frac{t-z_4}{t-z_0} ,\\
  &\phi_3 =\frac{(1-z_1)dt}{(t-1)(t-z_1)}=d\log \frac{t-z_4}{t-z_1} ,\quad  
  \phi_4 =\frac{(1-z_2)dt}{(t-1)(t-z_2)}=d\log \frac{t-z_4}{t-z_2} 
\end{align*}
form a basis of $H^1(Y;\na_0)$, and their intersection matrix is given as follows 
(cf. \cite{CM}):
\begin{align}
  \label{eq:int-mat-cohomology-Y}
  \tpi 
  \begin{pmatrix}
    \frac{1}{2c_4}+\frac{1}{2c_5} & \frac{1}{2c_4} & \frac{1}{2c_4} & \frac{1}{2c_4} \\
    \frac{1}{2c_4} & \frac{1}{2c_4}+\frac{1}{2c_0} & \frac{1}{2c_4} & \frac{1}{2c_4} \\
    \frac{1}{2c_4} & \frac{1}{2c_4} & \frac{1}{2c_4}+\frac{1}{2c_1} & \frac{1}{2c_4} \\
    \frac{1}{2c_4} & \frac{1}{2c_4} & \frac{1}{2c_4} & \frac{1}{2c_4}+\frac{1}{2c_2} 
  \end{pmatrix} .
\end{align}

\section{Hyperelliptic involution}\label{section:involution}
The pullback $\iota^{*}$ of differential forms induces involutions
\begin{align*}
  \iota^{*} :H^1(X;\na)\to H^1(X;\na),\quad 
  \iota^{*} :H^1(X;\na^{\vee})\to H^1(X;\na^{\vee}). 
\end{align*}
We also obtain involutions 
\begin{align*}
  \iota_{*} :H_1(X;\CL)\to H_1(X;\CL),\quad 
  \iota_{*} :H_1(X;\CL^{\vee})\to H_1(X;\CL^{\vee}),
\end{align*}
which are characterized as, for example,  
\begin{align*}
  \int_{\iota_{*}(\si)} T \vph =\int_{\si} T \cdot \iota^{*}\vph \quad 
  (\si \in H_1(X;\CL), \ \vph \in H^1(X;\na)). 
\end{align*}
As in \cite{G-Shibukawa} and \cite{G-Matsubara-Mitsui}, 
we decompose the twisted (co)homology groups into eigenspaces: 
\begin{align*}
  &H^1(X;\na) = H^{(1)} \op H^{(-1)} ,& 
  &H^1(X;\na^{\vee}) = H^{(1)\vee} \op H^{(-1)\vee},& \\
  &H_1(X;\CL)=H_{(1)} \op H_{(-1)} ,&  
  &H_1(X;\CL^{\vee})=H_{(1)}^{\vee} \op H_{(-1)}^{\vee}.&  
\end{align*}
Here, we set, for example, 
$H^{(\pm 1)}=\{ \vph \in H^1(X;\na) \mid \iota^{*}\vph =\pm \vph \}$. 
Further, these decompositions have orthogonalities with respect to 
the intersection forms and the period pairings. 
For example, we have 
\begin{align*}
  H_{(-1)} 
  &=\Big\{ \si\in H_1(X;\CL) ~\Big| ~ \int_{\si} T \vph =0 \ (\forall \vph \in H^{(1)})  \Big\} \\
  &=\{ \si\in H_1(X;\CL) \mid \langle \si ,\si' \rangle_{\h} =0 \ (\forall \si' \in H_{(1)}^{\vee}) \} .
\end{align*}

We study these eigenspaces by applying the results of \cite{G-Matsubara-Mitsui} 
with $G=\{ 1,\iota\} \simeq \Z/2\Z$.  
We set $f:=x(x-1)/y \in \CO_X (X)$, which satisfies $\iota^{*}f=-f$. 
This corresponds to $\al_{\chi}$ in \cite{G-Matsubara-Mitsui} for the nontrivial character $\chi \in\hat{G}$. 
Similarly to \cite{G-Matsubara-Mitsui}, we define 
the connections and local systems on $X$: 
\begin{align*}
  &\na_f :=\na +d\log f\we ,\quad 
  \na_f^{\vee} :=\na^{\vee} -d\log f\we ,\\
  &\CL_f:=\ker (\na^{\vee})=\C (Tf) ,\quad  
  \CL_f^{\vee}:=\ker (\na) =\C (Tf)^{-1} . 
\end{align*}
Note that we have $\CL_f \simeq \CL$ and $\CL_f^{\vee} \simeq \CL^{\vee}$ as local systems on $X$. 
The multivalued function $Tf$ is obtained by replacing $(c_0,\dots ,c_5)$ in the definition of $T$
with 
\begin{align}
  \label{eq:c-shift}
  (c_0',c_1',c_2',c_3',c_4',c_5') 
  := \left( c_0+\frac{1}{4},c_1-\frac{1}{4},c_2-\frac{1}{4},c_3-\frac{1}{4},c_4+\frac{1}{4},c_5+\frac{1}{4} \right) .
\end{align}
We also have the decompositions into eigenspaces:
\begin{align*}
  &H^1(X;\na_f) = H^{(1)}_f \op H^{(-1)}_f ,& 
  &H^1(X;\na_f^{\vee}) = H^{(1)\vee}_f \op H^{(-1)\vee}_f,& \\
  &H_1(X;\CL_f)=H_{(1),f} \op H_{(-1),f} ,&  
  &H_1(X;\CL_f^{\vee})=H_{(1),f}^{\vee}\op H_{(-1),f}^{\vee}.&  
\end{align*}
\begin{Fact}[{\cite{G-Matsubara-Mitsui}}]\label{fact:f-bai}
  \begin{enumerate}[(i)]
  \item We have two isomorphisms
    \begin{align*}
      f^{-1}\times :H^1(X;\na) \ni \vph \mapsto \frac{1}{f}\vph \in  H^1(X;\na_f) ,\quad 
      f\times :H^1(X;\na^{\vee}) \ni \vph \mapsto f\vph \in  H^1(X;\na_f^{\vee}) .
    \end{align*}
  \item We have two isomorphisms 
    \begin{align*}
      \ot f : H_1(X;\CL_f) \ni \vsi \mapsto \vsi \ot f \in H_1(X;\CL_f), \quad 
      \ot f^{-1} : H_1(X;\CL_f^{\vee}) \ni \vsi \mapsto \vsi \ot \frac{1}{f} \in H_1(X;\CL_f^{\vee}).
    \end{align*}
    Here, for $\vsi \in H_1(X;\CL_f)$, we construct $\vsi \ot f$ in the same manner as
    $\vsi$, loading it with the branch of $Tf$. 
    (For the precise definition, see \cite{G-Matsubara-Mitsui}.)
  \item Under the isomorphisms in (i) and (ii), we have the correspondence
    \begin{align*}
      H^{(-1)} \simeq H_f^{(1)},\quad  
      H^{(-1)\vee} \simeq H_f^{(1)\vee},\quad 
      H_{(-1)} \simeq H_{(1),f} ,\quad 
      H_{(-1)}^{\vee} \simeq H_{(1),f}^{\vee}. 
    \end{align*}
  \end{enumerate}
\end{Fact}
Recall that the double cover $\pr :X\to Y$ induces 
$\pr^{*}:H^1(Y;\na_0) \to H^1(X;\na)$, $\pr_{*}:H_1(X;\CL)\to H_1(Y;\CL_0)$, and so on. 
\begin{Fact}[{\cite{G-Matsubara-Mitsui}}]\label{fact:pr-eigen}
  By $\pr^{*}$ and $\pr_{*}$, we have
  \begin{align*}
    H^{(1)} \simeq H^1(Y;\na_0),\quad  
    H^{(1)\vee} \simeq H^1(Y;\na_0^{\vee}),\quad 
    H_{(1)} \simeq H_1(Y;\CL_0) ,\quad 
    H_{(1)}^{\vee} \simeq H_1(Y;\CL_0^{\vee}). 
  \end{align*}
  Moreover, we have
  \begin{align*}
    &\langle \pr^{*}\phi,\pr^{*}\phi' \rangle_{\ch}= 2 \cdot \langle \phi ,\phi'\rangle_{\ch,0} \quad 
      (\phi \in H^1(Y;\na_0),\ \phi' \in H^1(Y;\na_0^{\vee})) ,\\
    &\langle \vsi ,\vsi' \rangle_{\h} =\frac{1}{2} \cdot \langle \pr_{*}(\vsi),\pr_{*}(\vsi') \rangle_{\h,0} \quad 
      (\vsi \in H_{(1)},\ \vsi' \in H_{(1)}^{\vee}) .
  \end{align*}
\end{Fact}
By this fact, we can evaluate the intersection numbers on the ``$1$-eigenspaces''. 
To evaluate those on the ``$(-1)$-eigenspaces'', we combine Facts \ref{fact:f-bai} and \ref{fact:pr-eigen}. 
We define the connections and local systems on $Y$: 
\begin{align*}
  &\na_{0,f} :=\na_0 +d\log f\we ,\quad 
  \na_{0,f}^{\vee} :=\na_0^{\vee} -d\log f\we ,\\
  &\CL_{0,f}:=\ker (\na_0^{\vee})=\C (Tf) ,\quad  
  \CL_{0,f}^{\vee}:=\ker (\na_0) =\C (Tf)^{-1} . 
\end{align*}
Since $f$ is not a singlevalued function on $Y$, we note that $\CL_{0,f} \not\simeq \CL_0$. 
We also consider the intersection forms
\begin{align*}
  \langle \bu ,\bu \rangle_{\h,f} &: H_1(X;\CL_f) \times H_1(X;\CL_f^{\vee}) \longrightarrow \C ,\quad 
  \langle \bu ,\bu \rangle_{\h,0,f} : H_1(Y;\CL_{0,f}) \times H_1(Y;\CL_{0,f}^{\vee}) \longrightarrow \C ,\\
  \langle \bu ,\bu \rangle_{\ch,f} &: H^1(X;\na_f) \times H^1(X;\na_f^{\vee}) \longrightarrow \C ,\quad 
  \langle \bu ,\bu \rangle_{\ch,0,f} : H^1(Y;\na_{0,f}) \times H^1(Y;\na_{0,f}^{\vee}) \longrightarrow \C .
\end{align*}
These intersection forms can be evaluated by replacing the parameters $(c_0,\dots ,c_5)$
with $(c_0',\dots ,c_5')$ in (\ref{eq:c-shift}).
For $\vph \in H^{(-1)}$ and $\vph' \in H^{(-1)\vee}$, 
we have $\frac{1}{f}\vph \in H_f^{(1)}$ and $f \vph' \in H_f^{(1)\vee}$. 
Thus, there exist $\phi \in H^1(Y;\na_{0,f})$ and $\phi' \in H^1(Y;\na_{0,f}^{\vee})$ such that  
$\frac{1}{f}\vph =\pr^{*}\phi$ and $f \vph' = \pr^{*} \phi'$. 
We can evaluate the cohomology intersection number as follows: 
\begin{align}
  \label{eq:cohomology-int-(-1)eigen}
  \langle \vph ,\vph' \rangle_{\ch} 
  =\Big\langle \frac{1}{f}\vph ,f \vph' \Big\rangle_{\ch,f} 
  =\langle \pr^{*}\phi,\pr^{*}\phi' \rangle_{\ch,f}
  =2\cdot \langle \phi ,\phi' \rangle_{\ch ,0,f}. 
\end{align}
For $\vsi \in H_{(-1)}$ and $\vsi' \in H_{(-1)}^{\vee}$, we can also evaluate 
the intersection number in a similar manner: 
\begin{align}
  \label{eq:homology-int-(-1)eigen}
  \langle \vsi ,\vsi' \rangle_{\h}
  =\Big\langle \vsi \ot f,\vsi' \ot \frac{1}{f} \Big\rangle_{\h,f}
  =\frac{1}{2}\cdot \Big\langle \pr_{*}(\vsi \ot f),\pr_{*}\Big(\vsi' \ot \frac{1}{f} \Big) \Big\rangle_{\h,0,f}.
\end{align}

\section{Twisted cohomology group}\label{section:cohomology}
For $i=1,\dots ,4$, we set $\vph_i :=\pr^{*} \phi_i$ and $\vph_{i+4}:=f\cdot \vph_i$: 
\begin{align*}
  &\vph_1 =\frac{dx}{x-1}, \quad \vph_2 =\frac{dx}{x(x-1)} ,\quad 
  \vph_3 =\frac{(1-z_1)dx}{(x-1)(x-z_1)} ,\quad \vph_4 =\frac{(1-z_2)dx}{(x-1)(x-z_2)} ,\\
  &\vph_5 =\frac{xdx}{y}, \quad \vph_6 =\frac{dx}{y} ,\quad 
  \vph_7 =(1-z_1)\frac{xdx}{(x-z_1)y} ,\quad \vph_8 =(1-z_2)\frac{xdx}{(x-z_2)y} .
\end{align*}
It is obvious that $\vph_1 ,\dots ,\vph_4$ defines elements of $H^{(1)}$ or $H^{(1)\vee}$, 
and $\vph_5 ,\dots ,\vph_8$ defines elements of $H^{(-1)}$ or $H^{(-1)\vee}$.
\begin{Th}\label{th:cohomology-intersection}
  The intersection matrix $C$ of $\vph_1 ,\dots \vph_8$ is 
  \begin{align}
    \label{eq:cohomology-int-phi1-phi8}
    C =\tpi
    \begin{pmatrix}
       C(1) & O \\ O & C(-1)
    \end{pmatrix}, 
  \end{align}
  where 
  \begin{align*}
    C(1) &=
    \begin{pmatrix}
      \frac{c_4+c_5}{c_4c_5} & \frac{1}{c_4} & \frac{1}{c_4} & \frac{1}{c_4} \\
      \frac{1}{c_4} & \frac{c_4+c_0}{c_4c_0} & \frac{1}{c_4} & \frac{1}{c_4} \\
      \frac{1}{c_4} & \frac{1}{c_4} & \frac{c_4+c_1}{c_4c_1} & \frac{1}{c_4} \\
      \frac{1}{c_4} & \frac{1}{c_4} & \frac{1}{c_4} & \frac{c_4+c_2}{c_4c_2} 
    \end{pmatrix} ,\\
    C(-1) &=\frac{2}{z_1-z_2}\cdot 
    \begin{pmatrix}
      0& 0&-\frac{z_1}{(2c_1+\frac{1}{2})(z_1-z_3)} &\frac{z_2}{(2c_2+\frac{1}{2})(z_2-z_3)} \\ 
      0& 0&-\frac{1}{(2c_1+\frac{1}{2})(z_1-z_3)} &\frac{1}{(2c_2+\frac{1}{2})(z_2-z_3)} \\ 
      -\frac{z_1}{(2c_1-\frac{1}{2})(z_1-z_3)}& -\frac{1}{(2c_1-\frac{1}{2})(z_1-z_3)}&
      C_{33}&C_{34}\\ 
      \frac{z_2}{(2c_2-\frac{1}{2})(z_2-z_3)}& \frac{1}{(2c_2-\frac{1}{2})(z_2-z_3)}&
      C_{43}&C_{44}
    \end{pmatrix} ,\\
    &C_{33}=\frac{(1-z_1)z_1}{(2c_1-\frac{1}{2})(2c_1+\frac{1}{2})(z_1-z_3)}\left(
    \frac{2c_0-2c_1}{z_1}+\frac{2c_2+2c_1}{z_1-z_2} +\frac{2c_3+2c_1}{z_1-z_3}
    +\frac{2c_1+2c_4}{z_1-1}
    \right) ,\\
    &C_{34}=\frac{-1}{z_1-z_2}\Big(\frac{z_1(1-z_2)}{(2c_1-\frac{1}{2})(z_1-z_3)}
      +\frac{z_2(1-z_1)}{(2c_2+\frac{1}{2})(z_2-z_3)}\Big) ,\\
    &C_{43}=\frac{-1}{z_1-z_2}\Big(\frac{z_1(1-z_2)}{(2c_1+\frac{1}{2})(z_1-z_3)} 
      +\frac{z_2(1-z_1)}{(2c_2-\frac{1}{2})(z_2-z_3)} \Big) ,\\
    &C_{44}=-\frac{(1-z_2)z_2}{(2c_2-\frac{1}{2})(2c_2+\frac{1}{2})(z_2-z_3)}\left(
    \frac{2c_0-2c_2}{z_2}+\frac{2c_1+2c_2}{z_2-z_1} +\frac{2c_3+2c_2}{z_2-z_3}
    +\frac{2c_2+2c_4}{z_2-1}
    \right) .
  \end{align*}
\end{Th}
\begin{proof}
  Because of the orthogonality with respect to the intersection form, 
  it is clear that $C$ has the form of (\ref{eq:cohomology-int-phi1-phi8}). 
  The matrix $C(1)$ is obtained from (\ref{eq:int-mat-cohomology-Y}) by Fact \ref{fact:pr-eigen}. 
  To evaluate $C(-1)$, we use the equality (\ref{eq:cohomology-int-(-1)eigen}): 
  \begin{align*}
    \langle \vph_{4+i},\vph_{4+j} \rangle_{\ch}
    &=\langle f\cdot \vph_i,f\cdot \vph_j \rangle_{\ch}
      =\langle \vph_i,f^2\cdot \vph_j \rangle_{\ch,f}
      =\Big\langle \pr^{*}(\phi_i), \pr^{*}\Big(\frac{t(t-1)}{(t-z_1)(t-z_2)(t-z_3)}\cdot \phi_j \Big) \Big\rangle_{\ch,f} \\
    &=2\cdot \Big\langle \phi_i, \frac{t(t-1)}{(t-z_1)(t-z_2)(t-z_3)}\cdot \phi_j \Big\rangle_{\ch,0,f} .
  \end{align*}
  This intersection number can be evaluated in the same manner as in \cite{M-k-form}.
\end{proof}

By straightforward calculation, we have
\begin{align*}
  \det C(1)
  &=\frac{c_0 +c_1 +c_2 +c_4 +c_5}{c_0 c_1 c_2 c_4 c_5}=\frac{-c_3}{c_0 c_1 c_2 c_4 c_5} ,\\
  \det C(-1)
  &=\frac{2^8}{(z_1-z_2)^2 (z_1-z_3)^2 (z_2-z_3)^2 (4c_1-1)(4c_1+1)(4c_2-1)(4c_2+1)}, 
\end{align*}
which imply the following corollary. 
\begin{Cor}
  The $1$-forms $\vph_1,\dots ,\vph_4$ and $\vph_5,\dots ,\vph_8$ form a basis of $H^{(1)}$ 
  and $H^{(-1)}$, respectively. 
  In particular, $\vph_1,\dots ,\vph_8$ form a basis of $H^1(X;\na)$.
\end{Cor}
\begin{Rem}
  In terms of \cite[Example 2]{Watanabe-punctured}, 
  the holomorphic $1$-forms $\vph_5,\vph_6$ and 
  the logarithmic ones $\vph_1,\dots ,\vph_4$ form a basis of $E_{\infty}^{10}$, 
  and $\vph_7,\vph_8$ form a basis of $E_{\infty}^{01}$. 
\end{Rem}

\section{Twisted homology group}\label{section:homology}
\subsection{Twisted cycles}
We construct twisted cycles that represent elements of $H_1(X;\CL)$, 
referring to \cite{Mizutani-Watanabe-J}, \cite{Mizutani-Watanabe-E}. 
We take a lift $\widetilde{(1,\infty)} \subset X$ of the open interval $(1,\infty)\subset Y$ such that 
$(x,y)\in \widetilde{(1,\infty)}$ satisfies $y>0$. 
It can be naturally regarded as a locally finite twisted cycle, and we can construct 
a (finite) twisted cycle $\vsi_{45} \in H_1(X;\CL)$ such that
$\pr_{*}(\vsi_{45})=\si_{45}$. 
For $j=0,1,2,3$, we can also construct 
$\vsi_{j,j+1} \in H_1(X;\CL)$ satisfying $\pr_{*}(\vsi_{j,j+1})=\si_{j,j+1}$. 
Similarly to \cite{Mizutani-Watanabe-J}, \cite{Mizutani-Watanabe-E}, 
we consider an octagon associated with $\bar{X}$ (Figure \ref{fig:octagon}). 
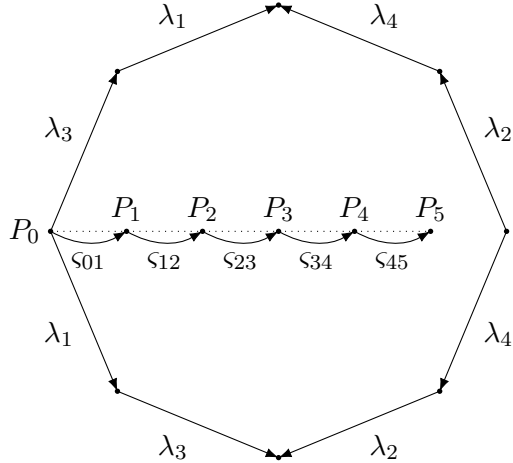
\begin{figure}[h]
  \centering
  \begin{tikzpicture}
    \coordinate (P00) at (3,0); 
    \coordinate (P01) at (2.12,2.12);
    \coordinate (P02) at (0,3);
    \coordinate (P03) at (-2.12,2.12);
    \coordinate[label=left:$P_0$] (P04) at (-3,0);
    \coordinate (P05) at (-2.12,-2.12);
    \coordinate (P06) at (0,-3);
    \coordinate (P07) at (2.12,-2.12);
    \coordinate[label=above:$P_1$] (P1) at (-2,0);
    \coordinate[label=above:$P_2$] (P2) at (-1,0);
    \coordinate[label=above:$P_3$] (P3) at (0,0);
    \coordinate[label=above:$P_4$] (P4) at (1,0);
    \coordinate[label=above:$P_5$] (P5) at (2,0);
    \foreach \P in {P00,P01,P02,P03,P04,P05,P06,P07,P1,P2,P3,P4,P5} \fill (\P) circle (1pt);
    \draw[->, >=Latex] (P00) -- node[above right]{$\lambda_2$} (P01);
    \draw[->, >=Latex] (P01) -- node[above right]{$\lambda_4$} (P02);
    \draw[->, >=Latex] (P03) -- node[above left]{$\lambda_1$} (P02);
    \draw[->, >=Latex] (P04) -- node[above left]{$\lambda_3$} (P03);
    \draw[->, >=Latex] (P04) -- node[below left]{$\lambda_1$} (P05);
    \draw[->, >=Latex] (P05) -- node[below left]{$\lambda_3$} (P06);
    \draw[->, >=Latex] (P07) -- node[below right]{$\lambda_2$} (P06);
    \draw[->, >=Latex] (P00) -- node[below right]{$\lambda_4$} (P07);
    \draw[dotted] (P04)--(P1) (P1)--(P2) (P2)--(P3) (P3)--(P4) (P4)--(P5);
    \draw[->, >=Latex] (P04) to [out = 330, in = 210] node[below]{$\varsigma_{01}$} (P1);
    \draw[->, >=Latex] (P1) to [out = 330, in = 210] node[below]{$\varsigma_{12}$} (P2);
    \draw[->, >=Latex] (P2) to [out = 330, in = 210] node[below]{$\varsigma_{23}$} (P3);
    \draw[->, >=Latex] (P3) to [out = 330, in = 210] node[below]{$\varsigma_{34}$} (P4);
    \draw[->, >=Latex] (P4) to [out = 330, in = 210] node[below]{$\varsigma_{45}$} (P5);
  \end{tikzpicture}
  \caption{Octagon associated with $\bar{X}$}\label{fig:octagon}
\end{figure}

In Figure \ref{fig:octagon}, the dotted line represents a branch cut of the multivalued function $T$. 
The cycles $\la_1,\dots, \la_4 \in H_1(\bar{X};\Z)$ do not define twisted cycles in $H_1(X;\CL)$, in general. 
By applying a suitable regularization, we can construct twisted cycles, which 
we also denote by $\la_j$ (see, for example, the construction of $l_0$ and $l_{\infty}$ in \cite{Mano-Watanabe}). 

In general, the twisted cycles defined in this section belong to neither 
$H_{(1)}$ nor $H_{(-1)}$. 
In the next section, we construct twisted cycles that lie in $H_{(1)}$ or $H_{(-1)}$.

\subsection{Eigenvectors}
In \cite{Mizutani-Watanabe-J}, \cite{Mizutani-Watanabe-E}, 
the integrals of holomorphic 1-forms on $\bar{X}$ over $\la_j$ are calculated. 
By suitably modifying the discussion, we can construct eigenvectors of $\iota_{*}$. 
\begin{Prop}[\cite{Mizutani-Watanabe-J}, \cite{Mizutani-Watanabe-E}]
  We set 
  \begin{align*}
    \la_{1,\pm}&:=\la_1 -(1\pm \bfe(c_0-2c_1-c_2))\vsi_{12}-(1\pm \bfe(c_0-2c_1-c_4))\vsi_{34},\\ 
    \la_{2,\pm}&:=\la_2 -(1\pm \bfe(c_0-2c_1-c_4))\vsi_{34},\\
    \la_{3,\pm}&:=\la_3 -(1\pm \bfe(c_0-c_1))\vsi_{01},\\
    \la_{4,\pm}&:=\la_4 -(1\pm \bfe(c_0-2c_1+c_3))\vsi_{23}.
  \end{align*}
  Then, we have 
  \begin{align}
    \label{eq:int-lambda-0}
    \int_{\la_{j,\pm}} T\vph_{\mp} =0 \quad (\forall\vph_{\mp} \in H^{(\mp 1)}),
  \end{align}
  and
  \begin{align}
    \label{eq:int-lambda-1}
    &\int_{\la_{1,\pm}}T \vph_{\pm}
      =\mp 2 \bfe(c_0-2c_1-c_2)\int_{\vsi_{12}}T \vph_{\pm} 
      \mp 2 \bfe(c_0-2c_1-c_4)\int_{\vsi_{34}}T \vph_{\pm} , \\
    \label{eq:int-lambda-2}
    &\int_{\la_{2,\pm}}T \vph_{\pm}
      =\mp 2 \bfe(c_0-2c_1-c_4)\int_{\vsi_{34}}T \vph_{\pm} , \\
    \label{eq:int-lambda-3}
    &\int_{\la_{3,\pm}}T \vph_{\pm}
      =\mp 2 \bfe(c_0-c_1)\int_{\vsi_{01}}T \vph_{\pm} ,\\
    \label{eq:int-lambda-4}
    &\int_{\la_{4,\pm}}T \vph_{\pm}
      =\mp 2 \bfe(c_0-2c_1+c_3)\int_{\vsi_{23}}T \vph_{\pm} , 
  \end{align}
  for any $\vph_{\pm} \in H^{(\pm 1)}$. 
\end{Prop}
\begin{proof}
  In Step 1 of \cite{Mizutani-Watanabe-J}, \cite{Mizutani-Watanabe-E}, 
  it is shown that the equality
  \begin{align}
    \label{eq:int-lambda3-1}
    \int_{\la_3} T \vph 
    &= \int_{\vsi_{01}} T\vph +\int_{-\iota(\vsi_{01})} T\vph \\
    \label{eq:int-lambda3-2}
    &= \int_{\vsi_{01}} T\vph +\bfe(c_0-c_1)\int_{\vsi_{01}} T\vph 
  \end{align}
  holds for a holomorphic 1-form $\vph$ on $\bar{X}$, by using the theta expression (\ref{eq:theta-expression}) of $T$. 
  In fact, we see that 
  the equality (\ref{eq:int-lambda3-1}) holds for any $\vph \in H^1(X;\na)$, while 
  (\ref{eq:int-lambda3-2}) holds for $\vph \in H^{(-1)}$, 
  since the proof requires only the property $\iota^{*}\vph =-\vph$. 
  By modifying the proof, we can also show 
  \begin{align*}
    \int_{\la_3} T \vph 
    = \int_{\vsi_{01}} T\vph -\bfe(c_0-c_1)\int_{\vsi_{01}} T\vph 
  \end{align*}
  for $\vph \in H^{(1)}$. 
  Therefore, we obtain (\ref{eq:int-lambda-0}) for $j=3$ and (\ref{eq:int-lambda-3}). 
  The other relations can also be proved by modifying Steps 2--4 of \cite{Mizutani-Watanabe-J}, \cite{Mizutani-Watanabe-E}. 
\end{proof}

By this proposition, we obtain the following theorem. 
\begin{Th}\label{th:sigma+-}
  We set 
  \begin{align*}
    \vsi_{01,\pm}
    &:=\mp \frac{1}{2 \bfe(c_0-c_1)} \la_{3,\pm}
    =\mp \frac{1}{2 \bfe(c_0-c_1)}\big( \la_3 -(1\pm \bfe(c_0-c_1))\vsi_{01} \big) ,\\
    \vsi_{23,\pm}
    &:=\mp \frac{1}{2 \bfe(c_0-2c_1+c_3)} \la_{4,\pm} 
    =\mp \frac{1}{2 \bfe(c_0-2c_1+c_3)}\big( \la_4 -(1\pm \bfe(c_0-2c_1+c_3))\vsi_{23} \big) ,\\
    \vsi_{34,\pm}
    &:=\mp \frac{1}{2 \bfe(c_0-2c_1-c_4)} \la_{2,\pm}
    =\mp \frac{1}{2 \bfe(c_0-2c_1-c_4)}\big( \la_2 -(1\pm \bfe(c_0-2c_1-c_4))\vsi_{34} \big) ,\\
    \vsi_{12,\pm}
    &:=\mp \frac{1}{2 \bfe(c_0-2c_1-c_2)} \big( \la_{1,\pm} -\la_{2,\pm} \big) \\
    &=\mp \frac{1}{2 \bfe(c_0-2c_1-c_2)} \big( \la_1 -\la_2 -(1\pm \bfe(c_0-2c_1-c_2))\vsi_{12} \big) .
  \end{align*}
  Then, $\vsi_{j,j+1,\pm}$ is an element of $H_{(\pm 1)}$ that satisfies 
  \begin{align}
    \label{eq:cycle-integral}
    \int_{\vsi_{j,j+1,\pm}}T\vph_{\pm} =\int_{\vsi_{j,j+1}} T\vph_{\pm} \quad (\forall\vph_{\pm} \in H^{(\pm 1)}).
  \end{align}
\end{Th}
\begin{proof}
  The equality (\ref{eq:int-lambda-0}) implies $\la_{j,\pm} \in H_{(\pm 1)}$. 
  By the equalities (\ref{eq:int-lambda-1})--(\ref{eq:int-lambda-4}), 
  we can show that the cycles $\vsi_{j,j+1,\pm}$ satisfy (\ref{eq:cycle-integral}). 
\end{proof}

\begin{Cor}
  We have 
  \begin{align*}
    &\iota_{*}(\vsi_{01})=-\frac{1}{\bfe(c_0-c_1)} \big( \la_3 -\vsi_{01} \big) ,&
    &\iota_{*}(\vsi_{23})=-\frac{1}{\bfe(c_0-2c_1+c_3)}\big( \la_4 -\vsi_{23} \big) ,& \\
    &\iota_{*}(\vsi_{34})=-\frac{1}{\bfe(c_0-2c_1-c_4)}\big( \la_2 -\vsi_{34} \big) ,&
    &\iota_{*}(\vsi_{12})=-\frac{1}{\bfe(c_0-2c_1-c_2)} \big( \la_1 -\la_2 -\vsi_{12} \big) .&
  \end{align*}
\end{Cor}
\begin{proof}
  We prove only the first equality, since the others are shown in a similar manner. 
  We set $\vsi_{01}':=-\frac{1}{\bfe(c_0-c_1)} \big( \la_3 -\vsi_{01} \big)$. 
  Since $\vsi_{01,\pm}=\frac{1}{2}(\vsi_{01} \pm \vsi_{01}')$, we have
  \begin{align*}
    \int_{\vsi_{01}'}T\vph_{\pm}
    =\pm \left( 2 \int_{\vsi_{01,\pm}}T\vph_{\pm} -\int_{\vsi_{01}}T\vph_{\pm} \right)
    =\pm \int_{\vsi_{01}}T\vph_{\pm}
    =\int_{\vsi_{01}}T \cdot \iota^{*}(\vph_{\pm})
    =\int_{\iota_{*}(\vsi_{01})}T \vph_{\pm}
  \end{align*}
  for any $\vph_{\pm}\in H^{(\pm 1)}$. 
  Therefore, we obtain $\vsi_{01}'=\iota_{*}(\vsi_{01})$. 
\end{proof}

\begin{Rem}\label{rem:sigma+-involution}
  As mentioned in the above proof, we have 
  $\vsi_{j,j+1,\pm}=\frac{1}{2}(\vsi_{j,j+1} \pm \iota_{*}(\vsi_{j,j+1}))$. 
\end{Rem}

Finally, we evaluate the intersection matrix. 
\begin{Th}
  The intersection matrix $H$ of $\vsi_{01,+},\dots,\vsi_{34,+},\vsi_{01,-},\dots,\vsi_{34,-}$ is 
  \begin{align}
    \label{eq:homology-int-sigma+-}
    H =\frac{1}{2}
    \begin{pmatrix}
       H(1) & O \\ O & H(-1)
    \end{pmatrix}, 
  \end{align}
  where 
  \begin{align*}
    H(\pm 1) &=
  \begin{pmatrix}
    \frac{1-\bfe(2c_0+2c_1)}{(1\pm \bfe(2c_0))(1\pm \bfe(2c_1))}&\frac{1}{1\pm \bfe(2c_1)}&0&0 \\
    \frac{\pm \bfe(2c_1)}{1\pm \bfe(2c_1)}&\frac{1-\bfe(2c_1+2c_2)}{(1\pm \bfe(2c_1))(1\pm \bfe(2c_2))}
    &\frac{1}{1\pm \bfe(2c_2)}&0 \\
    0&\frac{\pm \bfe(2c_2)}{1\pm \bfe(2c_2)}
    &\frac{1-\bfe(2c_2+2c_3)}{(1\pm \bfe(2c_2))(1\pm \bfe(2c_3))}&\frac{1}{1\pm \bfe(2c_3)} \\
    0&0&\frac{\pm \bfe(2c_3)}{1\pm \bfe(2c_3)}&\frac{1-\bfe(2c_3+2c_4)}{(1\pm \bfe(2c_3))(1\pm \bfe(2c_4))} 
  \end{pmatrix}.
  \end{align*}
\end{Th}
\begin{proof}
  Similarly to the proof of Theorem \ref{th:cohomology-intersection}, 
  we can show that $H$ has the form of (\ref{eq:homology-int-sigma+-}) and 
  $H(1)$ coincides with the intersection matrix (\ref{eq:int-mat-homology-Y}) of 
  $\si_{01},\dots ,\si_{34}$. 
  By (\ref{eq:homology-int-(-1)eigen}), we have 
  \begin{align*}
    \langle \vsi_{j,j+1,-},\vsi_{k,k+1,-} \rangle_{\h}
    &=\langle \vsi_{j,j+1,-}\ot f,\vsi_{k,k+1,-}\ot f^{-1} \rangle_{\h,f} \\
    &=\frac{1}{2} \langle \pr_{*}(\vsi_{j,j+1,-}\ot f), \pr_{*}(\vsi_{k,k+1,-}\ot f^{-1})\rangle_{\h,0,f} ,
  \end{align*}
  and the last intersection number can be obtained from $\langle \si_{j,j+1}, \si_{k,k+1}\rangle_{\h,0}$
  by replacing $(c_0,\dots ,c_5)$ with $(c_0',\dots ,c_5')$ in (\ref{eq:c-shift}).
  Thus, $H(-1)$ coincides with the matrix obtained from $H(1)$ by replacing $\bfe(2c_j)$ with $-\bfe(2c_j)$. 
\end{proof}
The determinants of the matrices $H(\pm 1)$ are 
\begin{align*}
  \frac{(1 \pm \bfe(-2c_5))}
  {(1 \pm \bfe(2c_0))(1 \pm \bfe(2c_1))(1 \pm \bfe(2c_2))(1 \pm \bfe(2c_3))(1 \pm \bfe(2c_4))}, 
\end{align*}
from which we obtain the following corollary. 
\begin{Cor}
  The twisted cycles $\vsi_{01,\pm},\vsi_{12,\pm},\vsi_{23,\pm},\vsi_{34,\pm}$ form a basis of $H_{(\pm 1)}$. 
  In particular, $\vsi_{01,+},\dots,\vsi_{34,+}$, $\vsi_{01,-},\dots,\vsi_{34,-}$ form a basis of $H_1(X;\CL)$.
\end{Cor}

\subsection{Linear relations}
We introduce some linear relations in $H_1(X;\CL)$, and 
reinterpret the relations given in \cite{Mizutani-Watanabe-J}, \cite{Mizutani-Watanabe-E} within our framework. 
\begin{Cor}
  The twisted cycles $\vsi_{01},\vsi_{12},\vsi_{23},\vsi_{34}$, $\la_{1},\la_{2},\la_{3},\la_{4}$
  are expressed by $\vsi_{01,+},\dots,\vsi_{34,+}$, $\vsi_{01,-},\dots,\vsi_{34,-}$ as follows: 
  \begin{align}
    \label{eq:sigma-sigma+-}
    &\vsi_{j,j+1} = \vsi_{j,j+1,+}+ \vsi_{j,j+1,-}   \qquad (j=0,1,2,3),\\
    \nonumber
    &\la_1 =(1- \bfe(c_0-2c_1-c_2))\vsi_{12,+} +(1+ \bfe(c_0-2c_1-c_2))\vsi_{12,-}\\
    \label{eq:lambda1-sigma+-}
    &\qquad 
    +(1- \bfe(c_0-2c_1-c_4))\vsi_{34,+} +(1+ \bfe(c_0-2c_1-c_4))\vsi_{34,-}, \\
    \label{eq:lambda2-sigma+-}
    &\la_2 =(1- \bfe(c_0-2c_1-c_4))\vsi_{34,+} +(1+ \bfe(c_0-2c_1-c_4))\vsi_{34,-}, \\
    \label{eq:lambda3-sigma+-}
    &\la_3 = (1- \bfe(c_0-c_1))\vsi_{01,+} +(1+ \bfe(c_0-c_1))\vsi_{01,-},\\
    \label{eq:lambda4-sigma+-}
    &\la_4 = (1- \bfe(c_0-2c_1+c_3))\vsi_{23,+} +(1+ \bfe(c_0-2c_1+c_3))\vsi_{23,-}.
  \end{align}
  Further, $\vsi_{01},\vsi_{12},\vsi_{23},\vsi_{34}$, $\la_{1},\la_{2},\la_{3},\la_{4}$ also form 
  a basis of $H_1(X;\CL)$. 
\end{Cor}
\begin{proof}
  The equality (\ref{eq:sigma-sigma+-}) follows from Remark \ref{rem:sigma+-involution}. 
  The equalities (\ref{eq:lambda1-sigma+-})--(\ref{eq:lambda4-sigma+-}) follow from 
  the definitions of $\vsi_{j,j+1,\pm}$ in Theorem \ref{th:sigma+-}. 
  It is easy to see that the matrix $Q$ satisfying 
  \begin{align*}
    (\vsi_{01},\vsi_{12},\vsi_{23},\vsi_{34},\la_{1},\la_{2},\la_{3},\la_{4})
    =(\vsi_{01,+},\dots,\vsi_{34,+}, \vsi_{01,-},\dots,\vsi_{34,-})Q ,
  \end{align*}
  has determinant $-2^4 \cdot \bfe(4c_0 -7c_1 -c_2+c_3-c_4)\neq 0$. 
  This implies that $\vsi_{01},\vsi_{12},\vsi_{23},\vsi_{34}$, $\la_{1},\la_{2},\la_{3},\la_{4}$ form 
  a basis. 
\end{proof}

\begin{Prop}[\cite{Mizutani-Watanabe-J}, \cite{Mizutani-Watanabe-E}]
  We have a relation in $H_1(X;\CL)$: 
  \begin{align*}
    &(1-\bfe(c_5-c_1))\la_1 
    +(\bfe(-c_0-2c_1+c_3)-\bfe(c_5-c_1))\la_2 \\
    &\qquad 
    +(\bfe(c_0+c_5)-\bfe(-c_0-c_1))\la_3
    +(\bfe(-c_0-2c_1-c_4)-\bfe(-c_0-2c_1-c_2))\la_4 \\
    &=(1-\bfe(-4c_0))\vsi_{01}
      +(1-\bfe(-4(c_0+c_1)))\vsi_{12}
      +(1-\bfe(-4(c_0+c_1+c_2)))\vsi_{23} \\
    &\qquad 
      +(1-\bfe(-4(c_0+c_1+c_2+c_3)))\vsi_{34}
      +(1-\bfe(4c_5))\vsi_{45}
 \end{align*}
\end{Prop}
\begin{proof}
  This proposition follows 
  from Proposition 2 and (11) of \cite{Mizutani-Watanabe-J}, \cite{Mizutani-Watanabe-E}. 
  Although the authors in \cite{Mizutani-Watanabe-J}, \cite{Mizutani-Watanabe-E} treat the integrals of 
  holomorphic 1-forms on $\bar{X}$, their calculations yield a relation between twisted cycles. 
\end{proof}

By considering the period pairings with $\vph_{\pm} \in H^{(\pm 1)}$, 
we obtain linear relations among the integrals. 
\begin{Cor}
  For $\vph_{\pm} \in H^{(\pm 1)}$, we have 
  \begin{align*}
    &(1-\bfe(4c_5))\int_{\vsi_{45}}  T\vph_{\pm} \\
    &=\Big( (\bfe(c_0+c_5)-\bfe(-c_0-c_1))(1\mp \bfe(c_0-c_1)) -(1-\bfe(-4c_0))\Big) \int_{\vsi_{01}}T\vph_{\pm} \\
    &+\Big( (1-\bfe(c_5-c_1))(1\mp \bfe(c_0-2c_1-c_2)) -(1-\bfe(-4(c_0+c_1)))\Big) \int_{\vsi_{12}}T\vph_{\pm} \\
    &+\Big( (\bfe(-c_0-2c_1-c_4)-\bfe(-c_0-2c_1-c_2))(1\mp \bfe(c_0-2c_1+c_3))
      -(1-\bfe(-4(c_0+c_1+c_2))))\Big) \int_{\vsi_{23}}T\vph_{\pm} \\
    &+\Big( (1+\bfe(-c_0-2c_1+c_3)-2\bfe(c_5-c_1))(1\mp \bfe(c_0-2c_1-c_4)) 
      -(1-\bfe(-4(c_0+c_1+c_2+c_3))))\Big) \int_{\vsi_{34}}T\vph_{\pm} .
  \end{align*}
\end{Cor}
\begin{Rem}
  If $\vph$ is a holomorphic 1-form on $\bar{X}$, then $\vph \in H^{(-1)}$, and 
  the above relation coincides with the Main Theorem of \cite{Mizutani-Watanabe-J}, \cite{Mizutani-Watanabe-E}\footnote{
    The last relation in \cite{Mizutani-Watanabe-J}, \cite{Mizutani-Watanabe-E} is not correct. 
    The right-hand side should be multiplied by $-1$. 
  }.
\end{Rem}

\begin{Ack}
  The author is grateful to Professor Humihiko Watanabe 
  for helpful discussions. 
  This work was supported by JSPS KAKENHI Grant Number JP24K06680.
\end{Ack}


\begin{thebibliography}{99}
\bibitem{AK}
  K. Aomoto and M. Kita, 
  ``Theory of Hypergeometric Functions'', 
  translated by K. Iohara, 
  Springer Monographs in Mathematics,
  Springer-Verlag, Tokyo, 2011. 

\bibitem{CM}
  K. Cho and K. Matsumoto, 
  Intersection theory for twisted cohomologies and 
  twisted Riemann's period relations. I, 
  \textit{Nagoya Math. J.} \textbf{139} (1995), 67--86.

\bibitem{Erdelyi}
  A. Erd\'elyi, 
  ``Higher transcendental functions, II'',
  McGraw-Hill Book Co., Inc., New York-Toronto-London, 1953.

  
\bibitem{G-RW-intersection}
  Y. Goto, 
  Intersection numbers of twisted homology and cohomology groups associated to the Riemann-Wirtinger integral, 
  \textit{Internat. J. Math.} \textbf{34} (2023), no. 03, 2350005, 32 pp.

\bibitem{G-Shibukawa}
  Y. Goto and G. Shibukawa, 
  Notes on twisted period relations for the Wirtinger integral; 
  arXiv:2511.17016. 

\bibitem{G-Matsubara-Mitsui}
  Y. Goto, S.-J. Matsubara-Heo and K. Mitsui, 
  Rapid decay homology intersection numbers of GKZ systems,
  in preparation. 



\bibitem{KY}
  M. Kita and M. Yoshida, 
  Intersection theory for twisted cycles,   
  \textit{Math. Nachr.} \textbf{166} (1994), 287--304. 



\bibitem{Mano-Watanabe}
  T. Mano and H. Watanabe, 
  Twisted cohomology and homology groups associated to the Riemann-Wirtinger integral,
  \textit{Proc. Amer. Math. Soc.} \textbf{140} (2012), no. 11, 3867--3881.  

\bibitem{M-k-form}
  K. Matsumoto, 
  Intersection numbers for logarithmic $k$-forms, 
  \textit{Osaka J. Math.} \textbf{35} (1998), 873--893. 

\bibitem{Mizutani-Watanabe-J}
  Y. Mizutani and H. Watanabe, 
  On a homogeneous linear relation satisfied by integrals on a Riemann surface of genus two, 
  \textit{Bouei Daigakko Rikougaku Kenkyuuhoukoku}, \textbf{58} (2) (2020), 23--30 (Japanese).

\bibitem{Mizutani-Watanabe-E}
  Y. Mizutani and H. Watanabe, 
  On a homogeneous linear relation satisfied by integrals on a Riemann surface of genus 2, 
  preprint. 


\bibitem{Pokraka-Ren-Rodriguez}
  A. Pokraka, L. Ren and C. Rodriguez, 
  A double copy from twisted (co)homology at genus $g$; 
  arXiv:2509.01598. 

\bibitem{Watanabe-elliptic-homology-cohomology}
  H. Watanabe, 
  Twisted homology and cohomology groups associated to the Wirtinger integral, 
  \textit{J. Math. Soc. Japan}, \textbf{59} (2007), no. 4, 1067--1080. 

\bibitem{Watanabe-wirtinger-diff-eq}
  H. Watanabe, 
  Linear differential relations satisfied by Wirtinger integrals, 
  \textit{Hokkaido Math. J.}, \textbf{38} (2009), no. 1, 83--95. 

\bibitem{Watanabe-punctured}
  H. Watanabe, 
  Twisted cohomology of a punctured Riemann surface, 
  \textit{Kumamoto J. Math.} \textbf{29} (2016), 55--63. 

\bibitem{Wirtinger}
  W. Wirtinger, 
  Zur Darstellung der hypergeometrischen Funktion durch bestimmte Integrale, 
  \textit{Akad. Wiss. Wien. S.-B. IIa} \textbf{111} (1902), 849--900. 

\end{thebibliography}
\end{document}